\documentclass[a4,12pt]{amsart}       % onecolumn (second format)

\usepackage[utf8]{inputenc}
\usepackage[T1]{fontenc}
\usepackage{yfonts}
\usepackage[english]{babel}
  \usepackage[normalem]{ulem}
\usepackage{lipsum}
\usepackage{amsmath}
\usepackage{amsthm}
\usepackage{amssymb}
\usepackage[shortlabels]{enumitem}
\usepackage{graphicx}
\usepackage{mathtools}
\usepackage{hyperref}
\usepackage{amsfonts}
\usepackage{latexsym}
\usepackage{amscd}

\usepackage[all]{xy}
\usepackage[dvipsnames]{xcolor}
\usepackage{stmaryrd}
\usepackage{tikz}
\usetikzlibrary{arrows.meta}

\DeclareMathOperator{\End}{End}
\DeclareMathOperator{\Hom}{Hom}

\DeclareMathOperator{\Path}{Path}

\DeclareMathOperator{\Sig}{\Sigma}

\def\Path{\text{Path}}
\def\Mod{\text{-Mod}}

\def\dualita#1#2{\mathrel{
                 \mathop{\vcenter{
                 \offinterlineskip
                 \hbox to 1.2truecm{$\mapsto$}%\kern ...
                 \hbox to 1.2truecm{$\mapsfrom$}}}%
                 }}
\usepackage{aliascnt}
\newtheorem{theorem}{Theorem}[section]
\newaliascnt{lemma}{theorem}
\newtheorem{lemma}[lemma]{Lemma}
\aliascntresetthe{lemma}

\newaliascnt{proposition}{theorem}
\newtheorem{proposition}[proposition]{Proposition}
\aliascntresetthe{proposition}

\newaliascnt{corollary}{theorem}
\newtheorem{corollary}[corollary]{Corollary}
\aliascntresetthe{corollary}

\newaliascnt{definition}{theorem}
\newtheorem{definition}[definition]{Definition}
\aliascntresetthe{definition}

\newaliascnt{remark}{theorem}
\newtheorem{remark}[remark]{Remark}
\aliascntresetthe{remark}
\newaliascnt{example}{theorem}
\newtheorem{example}[example]{Example}
\aliascntresetthe{example}

\usepackage{cleveref}
\crefname{theorem}{theorem}{theorems}
\Crefname{theorem}{Theorem}{Theorems}

\crefname{lemma}{lemma}{lemmas}
\Crefname{lemma}{Lemma}{Lemmas}

\crefname{proposition}{proposition}{propositions}
\Crefname{proposition}{Proposition}{Propositions}

\crefname{corollary}{corollary}{corollaries}
\Crefname{corollary}{Corollary}{Corollaries}

\crefname{definition}{definition}{definitions}
\Crefname{definition}{Definition}{Definitions}

\crefname{remark}{remark}{remarks}
\Crefname{remark}{Remark}{Remarks}

\crefname{example}{example}{examples}
\Crefname{example}{Example}{Examples}
\newcommand{\jj}{\hat{\text{\itshape\j}}}

\begin{document}
\title[injective envelopes of simple left ideals]{The injective envelope of simple modules over Leavitt path algebras I: \\  simple left ideals}
\author{Gene Abrams}
\address{Department of Mathematics, University of Colorado, Colorado
Springs, CO 80918 U.S.A., Orcid https://orcid.org/0000-0001-9046-9811}
\email{abrams@math.uccs.edu}
\author{Francesca Mantese}
\address{Dipartimento di Informatica, Universit\`{a} degli Studi di Verona, I-37134 Verona, Italy, Orcid https://orcid.org/0000-0002-2126-9426}
\email{francesca.mantese@univr.it}
\author{Alberto Tonolo}
\address{Dipartimento di Matematica ``Tullio Levi-Civita'', Universit\`{a} degli Studi di Padova, I-35121, Padova, Italy, Orcid https://orcid.org/0000-0002-9844-3998}
\email{alberto.tonolo@unipd.it}

\begin{abstract}
Let $K$ be any field and $E$ any directed graph.  We characterize up to isomorphism the simple (i.e., minimal) left ideals of the Leavitt path algebra $L_K(E)$. Then,
for each simple %(i.e., minimal) 
left ideal $I$ of %the Leavitt path algebra 
$L_K(E)$, we explicitly construct the injective envelope of $I$.  This result generalizes to all graphs $E$ and all simple left ideals in $L_K(E)$ the construction presented previously by the three authors for the specific case of the Jacobson algebra $R=K\langle X,Y | XY=1\rangle$ and the simple left $R$-ideal $R(1-YX)$.   Our method involves defining an $L_K(E)$-module structure on a $K$-vector space of infinite series.   We conclude the article by showing how our construction directly gives a description of the injective envelope of simple $L_K(E)$-modules arising from two types of infinite emitters in $E$.  
\end{abstract}

\maketitle

\section{Introduction}
%\B{  XXX  HERE IS A FIRST ATTEMPT AT AN INTRODUCTION FOR THE FIRST OF THE TWO ARTICLES.} 

The rings known as \emph{Leavitt path algebras}  were introduced two decades ago in \cite{AA05} and \cite{AMP07}.     Over the intervening twenty years, various aspects of  these rings (e.g., their multiplicative and ideal structure; their module-theoretic structure; $\mathbb{Z}$-graded structural aspects; connections between them and other topics including $C^*$-algebras and symbolic dynamics; etc.) have been widely and deeply studied.    
In the current article the three authors continue their investigation into the structure of the injective modules over  a Leavitt path algebra $L_K(E)$, where $E$ is an arbitrary directed graph and $K$ is any field.   

If $v$ is a vertex in $E$, then the left $L_K(E)$-ideal $L_K(E)v$ is simple (i.e., minimal) if and only if $v$ is a so-called {\it line point};  that is, no vertex $w$ for which there is a path from $v$ to $w$ admits either a bifurcation or a cycle based at it.  In particular, any sink in $E$ is trivially a line point.   As it turns out, every simple left $L_K(E)$-ideal is isomorphic to $L_K(E)v$ for some line point $v$ (see Proposition \ref{prop:simpleideallinepoint}).

In \cite{AMT21} the three authors explicitly identified the injective envelope of the simple left ideal $L_K(\mathcal{T})w$, where  $\mathcal{T} := \ \ \ \ \ \xymatrix{\bullet\ar@(ul,dl)%[]|{\,\,\,c}
\ar[r]&\bullet^w}$ (the ``Toeplitz'' graph).   This was  a nontrivial task; in the end, the description of this injective envelope involves the somewhat delicate notion of a ``formal series in paths ending at $w$'', see \cite[Corollary 29]{AMT21}.

In the current article, we extend this formal  series construction to give our main result, Theorem \ref{injhulltheorem}, in which we  explicitly construct the injective envelope of any simple left $L_K(E)$-ideal, for any graph $E$ and any field $K$.    (In the subsequent article \cite{Paper2} we construct the injective envelope of simple modules corresponding to some specified cycles, the so-called {\it exclusive} cycles, in any graph $E$.)

The article is arranged as follows.  
 In Section \ref{section:associativering}  we discuss some useful, general ring-theoretic results that are valid for  all rings with local units, and then  provide  the details undergirding the formal  series idea.  
 In Section~\ref{Lparesultssection}  we
 %we then recall for any Leavitt path algebras the notions of \emph{tree of a vertex}, \emph{hedhehog graph},
  present some  results about Leavitt path algebras which are central to the ideas contained in  the sequel.  In particular,   we show that any simple left ideal of $L_K(E)$ is isomorphic to a left ideal of the form $L_K(E)v$, where $v$ is a line point in $E$.  As well,  we describe the one-sided structure of those two-sided ideals of $L_K(E)$ that are generated by  subsets of  vertices of $E$. 
  
These initial sections put us in position to achieve in Section \ref{section:injenvline} our main result, Theorem \ref{injhulltheorem}, in which we describe, for any graph $E$ and any field $K$, the injective envelope of each simple left ideal of $L_K(E)$.  The notion of a formal series plays the key role.   The important contribution we make here is to show how this $K$-vector space of formal series admits  a left $L_K(E)$-module structure.  This module structure involves the (somewhat cumbersome) construction of a Cuntz-Krieger $E$-family of $K$-endomorphisms of this formal series $K$-space.  
  We finish the section by showing how to apply   Theorem~\ref{injhulltheorem}  to describe the injective envelope of two additional classes of simple modules given in \cite{AR14}.   
  
  We conclude the article with an Appendix, in which we present many of the germane computations verifying the Cuntz-Krieger property of the indicated $K$-endomorphisms.

\section[Some general ring-theoretic results]{Some general ring-theoretic results \\ for rings with local units}\label{section:associativering}
Let $S$ be an associative ring (we do not assume that $S$ has a multiplicative identity). 
%\B{More info about these rings can be found in ...}  

The ring $S$ is said to have a \emph{set of local units $A$} in case $A$ is a set of idempotents in $S$ having the property that every finite subset of $S$ is contained in a subring of the form $aSa$ for some $a\in A.$ 
%, for each finite subset $\{s_1,..., s_n\}$ of $S$, there exists $a\in A$ for which $as_ia=s_i$ for all $1\leq i\leq n$. In particular, we get
%\[as_i=a(as_ia)=a^2s_ia=as_ia=s_i\quad\text{and}\]
%\[ s_ia=(as_ia)a=as_ia^2=as_ia=s_i,\]
%justifying the name of such a property. 
The ring $S$ is said to have \emph{enough idempotents} in case there exists a set of nonzero orthogonal idempotents $H$ in $S$ for which the set $\Sig(H)$ of finite sums of distinct elements of $H$ is a set of local units for $S$. 

For a ring $S$ with local units, an abelian group $M$ is a left $S$-{\it module} if there is a (standard) left module action of $S$ on $M$, but with the added proviso that $M$ be {\it unitary}, i.e., that $SM=M$. (This is the appropriate generalization of the requirement that $1_S\cdot m=m$ for all $m$ in a left module $M$ over a unital ring $S$.) In particular, any left ideal of $S$ is a left $S$-submodule of the regular module ${}_SS$. If the condition $SM=M$ is missing, we say that $M$ is a left $S$-{\it premodule}.

For a ring $S$ with local units, the category $S$-Mod of (unitary) left $S$-modules and standard module homomorphisms retains many of the properties of a category of modules over a unital ring; importantly for us, that $S$-Mod has enough injectives, and that $S$-Mod admits direct limits.   

In this short section we present some results for rings with local units that will be very useful in the sequel. In particular,  we reformulate and prove a useful version of the Baer criterion \cite[Proposition 1.2]{AMT24} for rings with local units.

\begin{proposition}\label{prop:Baire}(\cite[Proposition 1.2]{AMT24}). Let $M$ be a left $S$-module and $I$ be a fixed left ideal of $S$. Assume that  any homomorphism $f : X \to M$ from any left
ideal $X\leq I$ extends to a homomorphism $\hat f:I\to M$, and also assume that any homomorphism $g:Y\to M$ from any left ideal $Y\geq I$ extends to a homomorphism $g:S\to M$. Then M is injective.
\end{proposition}
\begin{proof}
Let $J$ be any left ideal of $S$, and let $h:J\to M$ be an $S$-module homomorphism. The restriction $h_0:J\cap I\to M$ extends by assumption to $\hat h_0:I\to M$. Setting  
\[h_1(j+i)=h(j)+\hat h_0(i)\]
for each $j\in J, i\in I$,  yields  a homomorphism $h_1:J+I\to M$. That $h_1$ is well-defined follows from this observation:  $j+i=j'+i'$ implies $j-j'=i'-i\in J\cap I$, and hence 
%\small
\begin{align*}
    h(j)-h(j')&=h(j-j')=h_0(j-j')=h_0(i'-i)\\
    &=\hat h_0(i'-i)=\hat h_0(i')-\hat h_0(i),
\end{align*}
%\normalsize
from which $h(j)+\hat h_0(i)=h(j')+\hat h_0(i')$. Clearly $h_1$ extends $h$. Now, $h_1:J+I\to M$ extends by assumption to $\hat h_1:S\to M$. By the Baer Criterion for rings with local units \cite[Proposition~4.3.3]{AAM} we conclude that $M$ is injective.
\end{proof}

\begin{lemma}\label{lemma:restriction}
Let $M$ be a left $S$-module, $I\leq J$ fixed proper left ideals of $S$. Assume that the restriction map $\Hom_S(J,M)\to\Hom_S(I,M)$ is a monomorphism. If any homomorphism $I\to M$ extends to a homomorphism $S\to M$, then any homomorphism $J\to M$ also extends to a homomorphism $S\to M$.
\end{lemma}
\begin{proof}
Let $\varphi\in\Hom_S(J,M)$. By hypothesis, the restriction $\varphi_{\mid I}\in\Hom_S(I,M)$ extends to a homomorphism $\theta:S\to M$. Since the restrictions to $I$ of $\theta$ and $\varphi$ coincide, we have $\theta_{\mid J}=\varphi$ and hence $\theta$ extends $\varphi$.
\end{proof}

\begin{remark}\label{propertiespass} Let $S$ be a ring with local units  and $I$ a two-sided ideal of $S$. Then $S/I$ is a ring with local units.   Easily,  any left $S/I$-module $M$ becomes a left $S$-module through the projection $ S\to S/I$, by setting $s*x = (s+I)x$ for all $x\in M, s\in S.$\\
$\bullet$  \ It is straightforward to show that if $M$ is a simple left $S/I$-module, then $M$ is a simple left $S$-module.  \\
$\bullet$ \ It is similarly straightforward to show that if $L$ is an essential left $S/I$-submodule of $M$, then $L$ is also essential in $M$ as a left $S$-submodule.\\
$\bullet$ \  If $E$ is an injective left $S/I$-module, in general it is not guaranteed that $E$ is injective as left $S$-module (for instance $\mathbb Z/2\mathbb Z$  is an injective $\mathbb Z/2\mathbb Z$-module, but it is not an injective $\mathbb Z$-module). 
\end{remark}

With the final bullet point of the previous Remark as contrast, nevertheless we have
\begin{lemma}\label{lemma:injectivelocalunits}
Suppose $I$ is a two-sided ideal of the ring with local units $S$. Suppose $I$, viewed as a ring,  contains a set of local units. Then any injective left $S/I$-module is an injective left $S$-module.  

Consequently, in this situation, if $E$ is the injective envelope of the simple left $S/I$-module $M$ in $S/I$-Mod, then $E$ is the injective envelope of the simple left $S$-module $M$ in $S$-Mod.  
\end{lemma}
\begin{proof}
Let $E$ be an injective left $S/I$-module.    
Let $J$ be any left ideal of $S$, and $\varphi:J\to E$ an $S$-homomorphism.  Let $i\in I\cap J$; By the hypothesis of local units, we can find $i_0 \in I$ for which $i_0 i = i$.  Then $\varphi(i) = \varphi(i_0 i) = i_0 \varphi(i) = 0$ by the action of $S$ on $E$. Using this observation, the Baer criterion can be applied in the usual way to the left $S$-module $E$.

The second statement then follows directly from Remark \ref{propertiespass}. 
 %The Baer criterion can be applied in the usual way to the left $S$-module $E$.
%Using this observation,  setting $\overline\varphi(j+i)=\varphi(j)$ for each $j\in J$, $i\in I$, gives a well defined left $S$-module homomorphism$j_1+i_1=j_2+i_2$, then $j_1-j_2=i_2-i_1$, as follows:  if 
\end{proof}

\smallskip

We conclude this short section by making some observations about the notion of a vector space of formal series.  %\label{subsection:formalseries}

Let $K$ be a field. For each set $Y$, we denote by $K^Y$ the $K$-vector space of all functions $Y\to K$. Any element $\mathfrak p$ in $K^Y$ can be represented by a \emph{formal series in $Y$}, i.e., a (possibly infinite) sum
\[\sum_{y\in Y}k_yy.\]
The latter can be identified with the function $Y\to K$ that maps $\hat y$ to $k_{\hat y}$ for each $\hat y\in Y$; i.e., 
\[
\left(\sum_{y\in Y}k_yy\right)(\hat y)=k_{\hat y}\quad \forall \hat y\in Y.
\]
The set of formal series in $Y$ is denoted by $K[[Y]]$.
The role of variables $y\in Y$ is that of a ``placeholder" to identify the images of the elements of $Y$. The vector space structure of $K^Y$ is reflected in the usual operations of sum and scalar product  for the elements of $K[[Y]]$.
Having fixed any set $X$ and any injective map $f:Y\to X$, 
we can denote also by
\[
\sum_{y\in Y}k_yf(y)
\]
the function $Y\to K$ that maps $\hat y$ to $k_{\hat y}$ for each $\hat y\in Y$, where we use, instead of $y$, the expression $f(y)$ as placeholder.  More formally,
\[
\left(\sum_{y\in Y}k_yf(y)\right)(\hat y)=k_{\hat y}\quad \forall \hat y\in Y.
\]
%using the formal writing $yf(y)$ obtained juxtaposing $y$ and $f(y)$ as placeholder for $k_y$.
This seemingly-overly-cumbersome notation will serve us well in the sequel.

\section[Some general Leavitt path algebras results]{Some general results about \\ Leavitt path algebras  over arbitrary graphs }\label{Lparesultssection}

We assume the reader is familiar with the notion of a {\it Leavitt path algebra};  see \cite{AAM}  for additional details and background information.  We start this section by setting some notation.  

We let $E=(E^0, E^1, s,r)$ denote an arbitrary directed graph, $K$ any field,  and $L_K(E)$  the associated Leavitt path $K$-algebra.  $L_K(E)$ is the quotient of the path $K$-algebra $K\widehat E$ on the \emph{extended graph} $\widehat E$ by the ideal generated by the Cuntz-Krieger relations (see \cite[Definition 1.2.3]{AAM}). The algebra $L_K(E)$ always has enough idempotents: the vertices in $E^0$ are nonzero orthogonal idempotents, and the set of finite sums of distinct elements of $E^0$ constitutes a set of local units for $L_K(E)$. $L_K(E)$ has multiplicative identity if and only if the set $E^0$ of vertices is finite, in which case $1_{L_K(E)} = \sum_{v\in E^0}v.$

We denote by $\Path(E)$ the set of all paths in $E$, i.e., the finite sequences of edges $e_1\cdots e_m$ such that $r(e_i)=s(e_{i+1})$, $1\leq i<m$, together with the vertex set $E^0$.  
%Each vertex $v\in E^0$ is viewed as an element of $\Path(E)$ of length $0$.
We define a preorder $\geq$ on $E^0$ given by:
\[v\geq w\ \ \text{ in case there is }\mu\in\text{Path}(E) \ \mbox{for which} \  s(\mu)=v, r(\mu)=w.\]
If $H_0\subseteq E^0$ then \emph{the tree $T(H_0)$} is the set
\[T(H_0):=\{w\in E^0\mid \exists v\in H_0\text{ such that }v\geq w\}.\]
A vertex $v \in E^0$ is 
\begin{itemize}
    \item \emph{regular} if the set $s^{-1}(v)$ of edges with source $v$ has a finite cardinality greater than or equal to 1. 
\item  an \emph{infinite emitter} if $s^{-1}(v)$ is infinite. 
%$\text{Inf}(H)$ denotes the set of infinite emitters belonging to $H\subseteq E^0$.
\item  a \emph{sink} if $s^{-1}(v)=\emptyset$.
\end{itemize}
We denote the set of regular vertices (resp., infinite emitters) of $E$  by ${\rm Reg}(E)$ (resp., ${\rm Inf}(E)$).  
A set of vertices $H\subseteq E^0$ is 
\begin{itemize}
    \item \emph{hereditary} if whenever $v\in H$ and $w\in E^0$ for which $v\geq w$, then $w\in H$.
    \item \emph{saturated} if whenever a regular vertex $v$ has the property that $\{r(e)\mid e\in E^1, s(e)=v\}\subseteq H$, then $v\in H$.
\end{itemize}
It is easy to check that $T(H_0)$ is the hereditary closure of $H_0$.

A cycle $c=e_1\cdots e_n$ in $E$ is a closed path (i.e., $r(e_n)=s(e_1)$) such that $s(e_i)\not=s(e_j)$ for every $i\not=j$. We denote by $c^0$ the set of vertices belonging to the cycle $c$.
If $q(x)=k_0+k_1x+\cdots k_mx^m\in K[x]$, and $c$ is a cycle with source $s(c)$, we define
\[q(c):=k_0s(c)+k_1c+\cdots k_mc^m\in L_K(E).\]

\begin{lemma}\label{lemma:I_u}
The two-sided ideal $I(H_0)$ of $L_K(E)$ generated by any subset $H_0\subseteq E^0$ coincides with the two-sided ideal of $L_K(E)$ generated by its tree $T(H_0)$.
\end{lemma}
\begin{proof}
If $u\in T(H_0)$, then there exists a path $\mu\in\Path(E)$ such that $s(\mu)\in H_0$ and $r(\mu)=u$. Since $u=\mu^* s(\mu)\mu$ we have $u\in I(H_0)$. Thus $T(H_0)\subseteq I(H_0)$, and the result follows immediately. \end{proof}

\begin{definition}(compare with \cite[Definition~2.5.16]{AAM})\label{def:estesa}
Let $H$ be a nonempty hereditary subset of $E^0$. We denote by $F^+_E(H)$ the set
\[
\begin{aligned}
F^+_E(H)  \ :=  \  H \ \sqcup \ & \{\alpha\in \text{Path}(E)\mid  \ \alpha=e_1\cdots e_n,\text{ with }s(e_1)\in E^0\setminus H,\\
&r(e_i)\in E^0\setminus H\ \forall 1\leq i<n, \ \mbox{and} \  r(e_n)\in H\}.
\end{aligned}
\]
(So, in the notation of \cite[Definition~2.5.16]{AAM}, $F_E^+(H) := H \sqcup F_E(H).$)
\end{definition}

\begin{lemma}\label{lemma:mu*lambda}
Let $H$ be a nonempty hereditary subset of $E^0$.  
Then in $L_K(E)$,  for each $\mu, \lambda\in F^+_E(H)$ we have
\[\mu^*\lambda=\begin{cases}
    r(\mu)  & \text{if }\mu=\lambda, \\
      0& \text{otherwise}.
\end{cases}\]
\end{lemma}
\begin{proof}
By \cite[Lemma~1.2.12]{AAM} $\mu^*\lambda\not=0$ implies $\mu=\lambda\kappa$ for some $\kappa\in \text{Path}(E)$ or $\lambda=\mu\sigma$ for some $\sigma\in \text{Path}(E)$. If $\mu=\lambda\kappa$ and $\kappa\not=r(\lambda)$, then the ranges of two edges in $\mu$ would be in $H$, contrary to the property of the paths in $F^+_E(H)$ (see \Cref{def:estesa}). Analogously if $\lambda=\mu\sigma$ and $\sigma\not=r(\mu)$.
\end{proof}

The following result generalizes to an arbitrary hereditary set of vertices in an arbitrary Leavitt path algebra the property we observed  in \cite[Remark 5.1]{AMT21} in the case of the Leavitt path algebra of the Toeplitz graph. 
%he sink for the Jacobson algebra.
\begin{proposition}
\label{lemma:Rmu*}
The two-sided ideal $I(H)$ of $L_K(E)$ generated by a nonempty hereditary subset $H$ of $E^0$ is equal, as a left $L_K(E)$-ideal, to
\[ \sum_{\mu\in F^+_E(H)}L_K(E)\mu^* \ = \ \bigoplus_{\mu\in F^+_E(H)}L_K(E)\mu^*,\]
which in turn is isomorphic as left $L_K(E)$-modules to the external direct sum
\[ \bigoplus_{\mu\in F^+_E(H)}L_K(E)r(\mu).\]
\end{proposition}
\begin{proof}
Since $s(\mu^*)\in H$ for each $\mu\in F^+_E(H)$, it is clear that 
\[\sum_{\mu\in F^+_E(H)}L_K(E)\mu^*\leq I(H).\] By \cite[Lemma~2.4.1]{AAM}, any element of $I(H)$ has the form
\[\sum_{i=1}^nk_i\gamma_i\lambda_i^*\]
where $n\geq 1$, $k_i\in K$, $\gamma_i,\lambda_i\in\Path(E)$  such that $r(\gamma_i)=r(\lambda_i)\in H$.
If $\lambda_i$ has length $\geq 1$, set $\lambda_i=\ell_{i,1}\cdots \ell_{i,m_i}$ with $\ell_{i,j}\in E^1$ for each $1\leq j\leq m_i$, $1\leq i\leq n$. 
We have $r(\lambda_i)=r(\ell_{i,m_i})\in H$. 
If $s(\ell_{i,m_i})$ is not in $H$, then $s(\ell_{i,j})\notin H$ for each 
$1\leq j\leq m_i$ and hence $\lambda_i$ belongs to $F^+_E(H)$. 
Otherwise, if $s(\ell_{i,m_i})\in H$, let
\[\jj_i=\min\{j\mid 1\leq j\leq m_i, s(\ell_{i,j})\in H\}.\]
Consider the path
\[\hat\lambda_i=\begin{cases}
      \lambda_i  & \text{if }\lambda_i\text{ is a vertex,}\\
      \lambda_i& \text{if }s(\ell_{i,m_i})\notin H, \\
      s(\ell_{i,1})& \text{if }\jj_i=1, \\
     \ell_{i,1}\cdots \ell_{i,\jj_i-1} & \text{otherwise};
\end{cases}
\]
 and write $\lambda_i=\hat\lambda_i\lambda'_i$ for suitable paths 
$\lambda'_i$, for $1\leq i\leq n$. Then $\hat\lambda_i\in F^+_E(H)$ for $1\leq i\leq n$ and hence
\[\sum_{i=1}^nk_i\gamma_i\lambda_i^*=\sum_{i=1}^nk_i\gamma_i{\lambda'}_i^*\hat\lambda_i^*\in \sum_{\mu\in F^+_E(H)}L_K(E)\mu^*.\]
This establishes the first part of the Proposition.  
 
To show that the displayed sum of left ideals is direct, assume that $\sum_{i=1}^n r_i\mu_i^*=0$ with $r_i\in L_K(E)$ and $\mu_i\in F^+_E(H)$, $i=1,...,n$. Then by Lemma~\ref{lemma:mu*lambda} for each $i_0\in\{1,...,n\}$ we have
\[0=\left(\sum_{i=1}^n r_i\mu_i^*\right)\mu_{i_0}=r_{i_0}r(\mu_{i_0}),\]
and hence $r_{i_0}\mu_{i_0}^*=0$ as desired.  

For the final statement, we note that $L_K(E)\mu^*\cong L_K(E)r(\mu)$, as it is easy to check that 
\[L_K(E)\mu^*\to L_K(E)r(\mu) \ \ \mbox{via} \ \ \ell\mu^*\mapsto \ell\mu^*\cdot \mu=\ell r(\mu)\quad\text{and}
\]
\[ L_K(E)r(\mu)\to L_K(E)\mu^* \ \ \mbox{via} \ \  \ell r(\mu)\mapsto \ell r(\mu)\cdot\mu^*=
\ell\mu^*\]
are inverse isomorphisms.
\end{proof}

%\section{Minimal left ideals in a Leavitt path algebra, \\ and their relation to line points}\label{section:simpleleft}

We now give the formal definition of the vertices of a graph $E$ that will play the central role in this article.

\begin{definition} A vertex $v$ in the graph $E$ is called a \emph{line point} if there exist no bifurcations or cycles at any vertex of $T(v)$.
\end{definition}

By \cite[Proposition 2.6.11]{AAM}, for any graph $E$,  $v\in E^0$ has the property that $L_K(E)v$ is a  simple left ideal if and only if $v$ is a line point.

Vacuously, any sink is a line point.  In the graph
\[A_n \ \ \ := \ \ \  \underbrace{\bullet \to \bullet \to \cdots \to \bullet}_n 
\] 
all the vertices are line points; the same is true as well for all the vertices in  the graph
$$A_\infty \ \ \ := \ \ \  \bullet \to \bullet \to \bullet \to \cdots $$

\begin{definition}
     Let $v$ be a line point in $E$, and $w,w' \in T(v)$.   Then either $w\geq w'$ and there is a unique path $p$ in $E$ for which $s(p) = w$ and $r(p) = w'$, or otherwise there are no paths in $E$ for which $s(p) = w$ and $r(p) = w'$.  In the first situation, 
we denote the unique path from  $w$ to  $w'$ by  $$ p_{w,w'}.$$ 
\end{definition}

The following three lemmas will play a key role in \Cref{section:injenvline}.

\begin{lemma}\label{computewithpsubw}
Suppose $v$ is a line point, and $w,w',w'' \in T(v)$ for which $w\geq w'$ and $w' \geq w''$.   Then in $L_K(E)$, 
\begin{enumerate}
\item   $p_{w,w'} p_{w',w''} = p_{w,w''}$.
\item    $p_{w,w'} p_{w,w'}^* = w$.   
\end{enumerate}
\end{lemma}
\begin{proof}
The first statement is clear.  The second statement follows from the (CK2) relation, since no vertex in $T(v)$ has a bifurcation.   
\end{proof}

%Clearly $T(v)$ is an hereditary set. 

\begin{lemma}\label{lemma:semisimple}
The two-sided ideal $I(v)$ of $L_K(E)$ generated by the line point $v$ is a semisimple homogeneous  left $L_K(E)$-module.
\end{lemma}
\begin{proof}
%If $v$ is a line point and $w\in T(v)$ then $w=p_{v,w}^*vp_{v,w}\in I(v)$ and hence $I(v)=\langle T(v)\rangle$.
By Lemma~\ref{lemma:I_u}, and Proposition \ref{lemma:Rmu*} we have
\[I(v)=I(T(v))\cong \bigoplus_{\mu\in F^+_E(T(v))} L_K(E) r(\mu).\]
We conclude by observing that
\[L_K(E)v\to L_K(E)r(\mu) \ \ \  \mbox{via} \ \ \ \ell v\mapsto \ell vp_{v,r(\mu)}\quad\text{and}\]
\[L_K(E)r(\mu)\to L_K(E)v \ \ \ \mbox{via} \ \ \ \ell r(\mu)\mapsto \ell r(\mu) p_{v,r(\mu)}^*\]
are inverse isomorphisms. Hence $I(v)$ is isomorphic to the direct sum of copies of the simple left $L_K(E)$-module $L_K(E)v$.
\end{proof}

\begin{lemma}\label{lemma:leftidealLine}
The elements of the left ideal $L_K(E)v$ generated by the line point $v$ are of the form
\[\sum_{i=1}^nk_i\lambda_ip^*_{v,r(\lambda_i)} \ \ \ \ \ \mbox{where}   \ \lambda_i\in F^+_E(T(v)), k_i\in K.\]
\end{lemma}
\begin{proof}
Any element in $L_K(E)v$ is a $K$-linear combination of monomials $\gamma\delta^*$ with $\gamma,\delta\in\Path(E)$, $r(\gamma)=r(\delta)$, and $s(\delta)=v$. Since $T(v)$ is hereditary, $r(\delta)=r(\gamma)$ belongs to $T(v)$. 
Let us write $\gamma$ as a product $\lambda\gamma'$ with $\lambda\in F^+_E(T(v))$ and $s(\gamma')\in T(v)$. Since the vertices in  $\gamma'$ and $\delta$ are in $T(v)$ and $s(\delta)=r(\delta^*)=v$, by Lemma~\ref{computewithpsubw}
we get $\gamma'\delta^*=p^*_{v,s(\gamma')}=p^*_{v,r(\lambda)}$.
\end{proof}

Our goal for the remainder of this section is to establish \Cref{prop:simpleideallinepoint}, in which we show that every simple left ideal of $L_K(E)$ is necessarily isomorphic to $L_K(E)v$ for some line point $v \in E^0$.  

\begin{lemma}\label{Rp(c)isoRv}
Suppose $p(x) \in K[x] $ is not constant and has $p(0) \neq 0$.  
%, and is a non-unit in $K[x,x^{-1}]$, and 
Let $c$ be a cycle in the graph $E$ with $s(c) = v$.   Then $$L_K(E)p(c) \cong L_K(E)v$$ as left $L_K(E)$-modules.

\end{lemma}

\begin{proof}

It suffices to show that  right multiplication by $p(c)$ is injective on $L_K(E)v$; i.e., to show that if $r\in L_K(E)v$ has $rp(c) = 0$ then $r=0$.   

Since $r\in L_K(E)v$ we have $r=rv$. So, assume $rvp(c)=0$.   Multiplying by $-p(0)^{-1}$ if necessary, we may assume $p(c) = -v + c q(c)$ where $q(x) \in K[x]$.   Note that $c$ commutes with $q(c)$.
So, 
$$ r(-v + cq(c)) = 0 \ \ \Rightarrow \ \ rcq(c) = r  \ \ \Rightarrow  \ \  rc^2q(c)^2 = rcq(c) = r,$$
and continuing in this way, we get 
$$ rc^nq(c)^n = rv \ \ \ \mbox{for all } n\in \mathbb{N}.$$
But for sufficiently large $n$ the expression $rc^nq(c)^n$ is an expression in only real paths.   So $rv$ is an expression in only real paths, i.e. is an element of the path $K$-algebra $KE$.  By a degree argument the equation $rc^nq(c)^n = rv$ cannot hold in $KE$ for nonzero elements, whence $rv = 0$.  
\end{proof}

\begin{lemma}\label{vnotinRp(c)}

Suppose $p(x) \in K[x] $ is not constant and has $p(0) \neq 0$.  
%, and is a non-unit in $K[x,x^{-1}]$, and 
Let $c$ be a cycle in the graph $E$ with $s(c) = v$.   Then $v \notin L_K(E)p(c)$.   Consequently, $L_K(E)p(c)$ is a proper submodule of $L_K(E)v$.   
\end{lemma}

\begin{proof}

Multiplying by $-p(0)^{-1}$ if necessary, we may assume $p(c) = -v + c q(c)$ where $q(x) \in K[x]$.    Assuming $v = rp(c)$ for some $r\in L_K(E)$, we obtain a contradiction as follows.     Since $vp(c) = p(c)$ we may assume $r=rv$.  
$$ v = rp(c) \ \ \Rightarrow \ \ v = r (-v + cq(c)) \ \ \Rightarrow \ \ r=rv = -v + rcq(c).$$
Now substitute this expression for $r = rv$ into the right hand side, multiple times:
\begin{align*}
r = -v + (-v + rcq(c))(cq(c))  \ \ &\Rightarrow \ \ r = -v - cq(c) + rc^2q(c)^2\\
&\Rightarrow  \ \ r = -v - cq(c) - c^2q(c)^2 + rc^3 q(c)^3, 
\end{align*}
and continuing in this way, we get
$$r = -v - cq(c) - c^2q(c)^2 - \cdots  - c^{n-1}q(c)^{n-1}+ rc^n q(c)^n.$$
But $rc^n q(c)^n$ is in $KE$ for sufficiently large $n$.  Since the other summands in the displayed expression are already in the path $K$-algebra $KE$ we conclude that $r \in KE$.   As above, by a degree argument, this cannot happen in $KE$ unless $r = 0$.  
But then $v = rp(c) $ gives $v=0$, a contradiction.  
\end{proof}

\Cref{Rp(c)isoRv,vnotinRp(c)} immediately give 
\begin{lemma} \label{notsimple}
Suppose $p(x) \in K[x] $ is not constant and has $p(0) \neq 0$.  
%, and is a non-unit in $K[x,x^{-1}]$, and 
Let $c$ be a cycle in the graph $E$.   Then the left ideal $L_K(E)p(c)$ is not artinian (and so a fortiori is not simple).

Specifically, we have 
\[ L_K(E)p(c) \supsetneq L_K(E)p(c)^2 \supsetneq L_K(E)p(c)^3 \supsetneq \cdots\]

\end{lemma}

\begin{proposition}\label{prop:simpleideallinepoint}
Suppose $I$ is a simple left ideal of $L_K(E)$.   Then $I\cong L_K(E)v$ as left $L_K(E)$-modules for some line point $v$.
\end{proposition}

\begin{proof}

Write $I = L_K(E)x$ for some $0\neq x\in L_K(E)$.     By the Reduction Theorem \cite[Theorem 2.2.11]{AAM} there exist $\alpha, \beta \in {\rm Path}(E)$ for which either:

\smallskip
(1)  $\beta^* x \alpha = kv$ for some $0\neq k \in K$ and $v\in E^0$,  \ or

(2)   $\beta^* x \alpha = p(c)$ where $p(x) $ is a non-unit in $K[x,x^{-1}]$, and $c$ is a cycle (without exits) in $E$. 

\smallskip

We claim that when $L_K(E)x$ is simple, then $ L_K(E) \beta^* x \alpha \cong L_K(E)x$.   Since  $\beta^* x \alpha \neq 0$ we have $\beta^* x  \neq 0$, so that $L_K(E) \beta^* x$ is a nonzero submodule of $L_K(E)x$, whence $L_K(E)x = L_K(E)\beta^* x $. Now, define $\varphi: L_K(E)\beta^* x  \to L_K(E) \beta^* x \alpha$ by right multiplication by $\alpha$.   Since $\beta^* x \alpha \neq 0$ then $\varphi$ is not the zero map, so $\varphi$ is injective by the simplicity of $L_K(E)\beta^* x = L_K(E)x$.   And clearly $\varphi$ is surjective.

Using the claim, we conclude that the possibility (2) of the Reduction Theorem cannot occur here, as follows.  Suppose $\beta^* x \alpha = p(c)$ for some $\alpha, \beta$ paths in $E$,  where $p(x) $ is a non-unit in $K[x,x^{-1}]$.  Multiplying on the left by a power of $c^*$ or on the right by a power of $c$, and then by an appropriate nonzero constant in $K$,  we may assume that $p(c) = -v + cq(c)$ for some $0 \neq q(x) \in K[x]$.   But then the simple left $L_K(E)$-module $L_K(E)x$ has $L_K(E)x \cong L_K(E)\beta^*x\alpha$ (by the claim), which in turn equals  $L_K(E)p(c)$, which is not simple by Lemma \ref{notsimple}, a contradiction.  

Thus possibility (1) of the Reduction Theorem must hold, i.e., $\beta^* x \alpha = kv$ for some $0\neq k \in K$ and $\alpha, \beta \in {\rm Path}(E)$. Again using the claim $ L_K(E) \beta^* x \alpha \cong L_K(E)x$, so that $L_K(E)x \cong L_K(E)kv= L_K(E)v$. 

The result now follows from \cite[Proposition 2.6.11]{AAM}.
\end{proof}

Consequently, 

\begin{corollary}\label{linepointsufficient}
Let $E$ be an arbitrary graph, and $K$ any field.  In order to construct the injective envelope of any simple left $L_K(E)$-ideal, it suffices to construct the injective envelope of each simple left $L_K(E)$-ideal of the form $L_K(E)v$ for $v$ a line point in $E$. 
\end{corollary}

We complete exactly this sufficient task in our main result, Theorem \ref{injhulltheorem}.

\section{The injective envelope of simple left $L_K(E)$-ideals}\label{section:injenvline}

Let $E$ be an arbitrary graph, $K$ any field,  and  $L_K(E)$ the associated Leavitt path algebra.
By Proposition~\ref{prop:simpleideallinepoint} any simple left $L_K(E)$-ideal is isomorphic to a left $L_K(E)$-ideal generated by some line point.
\begin{center}
\emph{We assume throughout this section that $v$ denotes a line point in $E$.} 

\smallskip

\emph{Throughout this section $L_K(E)$ will often be denoted simply by $R$.}

\end{center}
%Denote by $R$ the $K$-algebra $L_K(E)$ and by $I(v)$ the two sided ideal of $R$ generated by $v$. 
 
We seek to explicitly describe the injective envelope of the simple left $R$-module $Rv$.  We achieve this goal in three steps.  First, we build a $K$-vector space inside which $Rv$ lives as a subspace (Definition \ref{def:sourcefinite}).  The form of the elements of this larger vector space will be motivated by  the form of the elements of $Rv$ as presented in Lemma  \ref{lemma:leftidealLine}.   Second, we expend some effort to show that this larger vector space can be endowed with an $R$-module structure (Proposition  \ref{prop:formalisRmod}), and then somewhat  easily show that this $R$-module contains $Rv$ as an essential submodule (Lemma \ref{L(E)vessentialinMv}).   We conclude by showing that this  overmodule of $Rv$ is in fact injective (Theorem \ref{injhulltheorem}). \\    

\subsection*{First step}  
Using Lemma \ref{lemma:mu*lambda}, we see that  the map $F^+_E(T(v))\to \Path(\widehat E)$, $\mu\mapsto \mu p^*_{v,r(\mu)}$ is injective.
So, following what we observed at the end of Section \ref{section:associativering} regarding formal series,  we give the following

\begin{definition}\label{Mvdef}
For a line point $v$ in the graph $E$, we denote by 
$$K[[F^+_E(T(v))]]$$ the $K$-vector space of all functions from $F^+_E(T(v))$ to $K$. We represent any element $\mathfrak p$ in $K[[F^+_E(T(v))]]$ with the following ``formal series notation'':
\[\mathfrak p=\sum_{\mu\in F^+_E(T(v))}k_\mu \mu p^*_{v,r(\mu)}\qquad k_\mu\in K,\]
where $k_\mu=\mathfrak p(\mu)=\left(\sum_{\mu\in F^+_E(T(v))}k_\mu \mu p^*_{v,r(\mu)}\right)(\mu)$.
\end{definition}

%\begin{remark}\label{rem:formser}
  %  Observe that each $\mu\in F^+_E(T(v))$  uniquely determines the ghost path $p^*_{v,r(\mu)}$. Thus, it is not inappropriate to use $\mu p^*_{v,r(\mu)}$  as a ``placeholder" for the scalar $k_\mu$. This notation allows us to naturally view $Rv$ as a $K$-subspace of the $K$-vector space $K[[F^+_E(T(v))]]$ (see \Cref{lemma:leftidealLine}).
%\end{remark}

\begin{definition}\label{def:sourcefinite}
A formal series $\mathfrak p$ in $K[[F^+_E(T(v))]]$ is \emph{source finite} if the set
 $S(\mathfrak p):=\{s(\mu):\mathfrak p(\mu)\not=0\}$ is finite. 
We denote by \[K[[F^+_E(T(v))]]^{sf}\] the $K$-vector subspace of $K[[F^+_E(T(v))]]$ consisting of all source finite formal series.
 \end{definition}
 For any element $\mathfrak p$ in
 $K[[F^+_E(T(v))]]^{sf}$, denote by  $s(\mathfrak p)$ the finite sum 
\[s(\mathfrak p):=\sum_{u\in S(\mathfrak p)}u.\]
Clearly $s(\mathfrak p)$ is an idempotent in $R$. 

Easily,  the set $\{ u\in E^0 \ | \ u \geq w, \mbox{ for some } w\in T(v)\}$ is finite if and only if   $K[[F^+_E(T(v))]]=K[[F^+_E(T(v))]]^{sf}$.  In particular, this equality holds in case $E$ is a finite graph.

%If  the set of vertices $u$ in $E^0$ such that $u\geq w$ for some $w\in T(v)$ is finite (in particular $T(v)$ has to be finite), then  $K[[F^+_E(T(v))]]=K[[F^+_E(T(v))]]^{sf}$.\\

The elements of $Rv$ are clearly source finite: they are the formal series 
\[\mathfrak p=\sum_{\mu\in F^+_E(T(v))}k_\mu \mu p^*_{v,r(\mu)}\in K[[F^+_E(T(v))]]\]
for which $k_\mu=0$ for almost all $\mu\in F^+_E(T(v))$ (see \Cref{lemma:leftidealLine}). In particular $v = 1_K v p_{v,v}^*$ is the source finite formal series $\mathfrak p$ such that $\mathfrak p(v)=1_K$ and $\mathfrak p(\mu)=0$ for each $\mu\in F^+_E(T(v))\setminus\{v\}$.

\medskip

\subsection*{Second step}
We want now to prove that the $K$-vector space $K[[F^+_E(T(v))]]^{sf}$ is a left $R$-module that contains the simple left $R$-module $Rv$ as an essential submodule.

\begin{proposition}\label{prop:formalisRmod}
The $K$-vector space $K[[F^+_E(T(v))]]$ is a left $R$-premodule, while $K[[F^+_E(T(v))]]^{sf}$ is a left $R$-module. Indeed,
\[K[[F^+_E(T(v))]]^{sf}=R\cdot K[[F^+_E(T(v))]].
\]
\end{proposition}
\begin{proof}
We claim that there is Cuntz-Krieger $E$-family in ${\rm End}_K(K[[F^+_E(T(v))]])$.   We establish the claim in the Appendix (Section \ref{sec:appendix}), by explicitly defining the elements $ P_u, S_e, S_{e^*}$ ($u\in E^0, e\in E^1$) of 
${\rm End}_K(K[[F^+_E(T(v))]])$, and then demonstrating that this specific collection forms a  Cuntz-Krieger $E$-family. 

The existence of a Cuntz-Krieger $E$-family in ${\rm End}_K(K[[F^+_E(T(v))]])$ then guarantees the existence of a ring homomorphism $\Phi: R \to {\rm End}_K(K[[F^+_E(T(v))]])$, which in turn (by standard ring theory) endows $K[[F^+_E(T(v))]]$ with an $R$-premodule structure by setting $$x \cdot m := \Phi(x)(m)$$
for $x\in R$ and $m\in K[[F^+_E(T(v))]]$.  

As one can see from their explicit descriptions as given in the Appendix, the endomorphisms $ P_u, S_e, S_{e^*}$ ($u\in E^0, e\in E^1$) have the effect, respectively, of  multiplying by $u$, $e$, and $e^*$ each of the individual terms of the formal series $\mathfrak p=\sum_{\mu\in F^+_E(T(v))}k_\mu \mu p^*_{v,r(\mu)}$.   
More formally,

$$P_u(\sum_{\mu\in F^+_E(T(v))}k_\mu \mu p^*_{v,r(\mu)}) = \sum_{\mu\in F^+_E(T(v))}k_\mu u\mu p^*_{v,r(\mu)} \ \ \forall \ u \in E^0,$$

$$S_e(\sum_{\mu\in F^+_E(T(v))}k_\mu \mu p^*_{v,r(\mu)}) = \sum_{\mu\in F^+_E(T(v))}k_\mu e\mu p^*_{v,r(\mu)} \ \ \forall \ e \in E^1, $$

$$S_{e^*}(\sum_{\mu\in F^+_E(T(v))}k_\mu \mu p^*_{v,r(\mu)}) = \sum_{\mu\in F^+_E(T(v))}k_\mu e^*\mu p^*_{v,r(\mu)} \ \ \forall \ e \in E^1.$$

%Multiplying by $u$, $e$, and $e^*$ all summands of the formal series on $\mathfrak p=\sum_{\mu\in F^+_E(T(v))}k_\mu\cdot \mu p^*_{v,r(\mu)}$, we endow $K[[F^+_E(T(v))]]$ with a left $R$-premodule structure. In \Cref{sec:appendix} the latter claim is explained in detail.
% Let us prove that there exists a $K$-algebra homomorphism
% \[\Phi:R\to \End_K(K[[F^+_E(T(v))]]).\]
% \textbf{Claim}:
% The subset $\{P_v, S_e, S_{e^*}\mid v\in E^0, e\in E^1\}$ forms a Cuntz Krieger $E$-family in $\End(K[[F^+_E(T(v))]])$ (see \Cref{sec:appendix} for a detailed proof).
%   Consequently,  by the Universal Property of $L_K(E)$ (see, e.g., \cite[Remark 1.2.5]{AAM}),  there exists a $K$-algebra homomorphism 
%   \[\Phi:  R \to  \End(K[[F^+_E(T(v))]]),\]  which thereby endows $K[[F^+_E(T(v))]]$ with the structure of left $R$-premodule.    Specifically, for $r\in R$ and $m\in K[[F^+_E(T(v))]]$,  $$r\cdot m := \Phi(r)(m).$$
We now consider the $K$-subspace $K[[F^+_E(T(v))]]^{sf}$ of $K[[F^+_E(T(v))]]$.  We claim that $K[[F^+_E(T(v))]]^{sf}$ is closed under the premodule action of $R$ on  $K[[F^+_E(T(v))]]$ defined above, and that under this action $K[[F^+_E(T(v))]]^{sf}$ is in fact unitary. 
Because $R\cdot R = R$, to establish the claim it suffices to show that 
 $$R\cdot K[[F^+_E(T(v))]]=K[[F^+_E(T(v))]]^{sf}.$$ Let $u\in E^0$ and  $\mathfrak p = \sum_{\mu\in F^+_E(T(v))}k_\mu \mu p^*_{v,r(\mu)} \in K[[F^+_E(T(v))]]$.   Then using the action defined above, we get
 $$ u \cdot \mathfrak p = \sum_{\mu\in F^+_E(T(v)), \ s(\mu) = u}k_\mu \mu p^*_{v,r(\mu)} .$$
 \noindent 
This shows that for each $u\in E^0$ and $\mathfrak p \in K[[F^+_E(T(v))]]$ we have 
$$u \cdot \mathfrak p = 0 \ \ \Leftrightarrow \ \ u \notin S(\mathfrak p).$$
 Let $\mathfrak q \in K[[F^+_E(T(v))]]^{sf}$, 
 and write   $\mathfrak q = \sum_{\nu\in F^+_E(T(v))}k_\nu \nu p^*_{v,r(\nu)}.$ 
  Since $\mathfrak q$ is source finite  the set $S(\mathfrak q):= \{s(\nu) : k_\nu\neq 0\}$ is finite; as defined above, let $s(\mathfrak q)$ denote the idempotent $\sum_{u\in S(\mathfrak q)}u$ of $R$.   Then  it is easy to show that $s(\mathfrak q) \cdot \mathfrak q = \mathfrak q$, which yields \[R\cdot K[[F^+_E(T(v))]] \ \supseteq \ R\cdot K[[F^+_E(T(v))]]^{sf}\ \supseteq \ K[[F^+_E(T(v))]]^{sf}.\]
 Conversely,  let $r\in R$ and $\mathfrak p\in   K[[F^+_E(T(v))]]$.
   %;  write $m = \sum_{\mu\in F^+_E(T(v))}k_\mu \mu p^*_{v,r(\mu)}.$  
 %Again invoking the action described above, 
 %$$m = \sum_{\mu\in F^+_E(T(v))}k_\mu \mu p^*_{v,r(\mu)}.$$
 Since the set of finite sums of distinct vertices of $E$ forms a set of local units for $R$, there exist vertices $v_1, \dots , v_n$ in $E^0$ for which $(\sum_{i=1}^n v_i)  r = r$.   In particular if $u \in E^0 \setminus \{v_1, \dots , v_n\}$ then $ur=0$.   But then for any such $u$ we have $u\cdot (r\cdot \mathfrak p) = (ur) \cdot \mathfrak p = 0$. Thus, by the previous observation $u\notin S(r\cdot \mathfrak p)$, so that $S(r \cdot \mathfrak p) \subseteq \{v_1, \dots , v_n\}$ is finite, and thus $r\cdot \mathfrak p \in K[[F^+_E(T(v))]]^{sf}.$
 \end{proof}

  For notational convenience, we denote by $X_v^E$ the left $L_K(E)$-module
$$X_v^E:= K[[F^+_E(T(v))]]^{sf},$$
and we call $X_v^E$ the {\it formal series left $L_K(E)$-module associated to the line point $v$}.

\begin{lemma}\label{L(E)vessentialinMv} Let $v$ be a line point in the graph $E$.
Then $Rv$ is an essential submodule of the left $R$-module 
%$X_v^E$.
$X_v^E$. 
\end{lemma}
\begin{proof} 
We already observed (see Lemma \ref{lemma:leftidealLine}) that $Rv\subseteq  X_v^E$. Moreover, using the comment made before Proposition~\ref{prop:formalisRmod}, we see that the left $R$-module structure of $X_v^E$ makes $Rv$ a left $R$-submodule of $X_v^E$.

Now let $\mathfrak p = \sum_{\mu\in F^+_E(T(v))} k_\mu\mu p^*_{v,r(\mu)}$ be a nonzero  formal series in $X_v^E$.
%;  specifically, the series is source finite.   
Pick $\gamma \in F^+_E(T(v))$ with $k_\gamma \neq 0$.   Multiplying $\mathfrak p$ on the left by $k_\gamma^{-1} p_{v,r(\gamma)}  \gamma^*$, by Lemma~\ref{lemma:mu*lambda} and Lemma~\ref{computewithpsubw} we get 
\[k_\gamma^{-1} p_{v,r(\gamma)}  \gamma^*\mathfrak p = p_{v,r(\gamma)}p_{v,r(\gamma)}^* = v\in Rv.\qedhere \]
\end{proof}

\medskip

\subsection*{Third step}
We now have all the pieces in place  to establish the main result of the article.  

\begin{theorem}\label{injhulltheorem} Let $E$ be any graph, and $K$ any field.    Let $v$ be a line point in  $E$. 
Then the formal series left $L_K(E)$-module $X_v^E:= K[[F^+_E(T(v))]]^{sf}$  is the injective envelope of the simple left $L_K(E)$-ideal $L_K(E)v$. 
\end{theorem}
\begin{proof} We continue to denote $L_K(E)$ by $R$.   By Lemma~\ref{L(E)vessentialinMv} we know that $Rv$ is essential in $X_v^E$.  Thus we need only establish that $X_v^E$ is an injective left $R$-module.   We use  Proposition~\ref{prop:Baire}, applied here to the one-sided ideal structure of the two-sided ideal $I(v)$ of $R$ generated by the vertex $v$.   So it suffices to show that:

\smallskip

  (1) If $C$ is a left ideal of $R$ contained in $I(v)$ then any $R$-homomorphism from $C$ to $X_v^E$ extends to an $R$-homomorphism from $I(v)$ to $X_v^E$, and 
  
  (2) if $D$ is any left ideal of $R$ containing $I(v)$ then any $R$-homomorphism from $D$ to $X_v^E$ extends to an $R$-homomorphism from $R$ to $X_v^E$.  

\smallskip

But (1) is immediate since, by Lemma~\ref{lemma:semisimple},  $I(v)$ is semisimple as a left $R$-module, so any left $R$-submodule of $I(v)$  is a direct summand of $I(v)$, and so any such map extends trivially.

We establish (2) in two steps as follows. 

First, we show that the result holds for $D = I(v)$.  By Lemma~\ref{lemma:semisimple}, $I(v)$ coincides with $I(T(v))$. By Proposition~\ref{lemma:Rmu*}, any $R$-homomorphism $\varphi:I(v)\to X_v^E$ is uniquely determined by the values of $\varphi(\mu^*)$, where $\mu \in F^+_E(T(v))$.  Since $\varphi(\mu^*)=\varphi(r(\mu)\mu^*)=r(\mu)\varphi(\mu^*)$, and $r(\mu)\in F^+_E(T(v))$, by Lemma~\ref{lemma:mu*lambda} we have 
\[\varphi(\mu^*)=k_\mu p_{v,r(\mu)}^*\]
for suitable $k_\mu\in K$.
 Now, define $\epsilon_\varphi \in K[[F^+_E(T(v))]]$ by setting  
$$\epsilon_\varphi := \sum_{\lambda \in F^+_E(T(v))} k_\lambda \lambda p^*_{v,r(\lambda)}.$$

By Proposition \ref{prop:formalisRmod}, 
%using the previously established result $R K[[F^+_E(T(v))]] = X_v^E$, 
% the right product by $\epsilon_\varphi$ defines an $R$-homomorphism
\[\theta_{\epsilon_\varphi}: R \to X_v^E,\quad r\mapsto r\cdot \epsilon_\varphi\]
is an $R$-homomorphism.
%For each $\hat r\in R$, there exists a finite number of vertices $V_{\hat r}\subseteq E^0$ such that $\sum_{u\in V_{\hat r}}u\cdot\hat r=\hat r$. Therefore, also if $\epsilon_\varphi$ belongs to the non unital left $R$-module $K[[F^+_E(T(v))]]$, we have $\hat r\cdot \epsilon_\varphi \in K[[F^+_E(T(v))]]^{sf}$. Thus we can define $\theta: R \to K[[F^+_E(T(v))]]^{sf}$ via right multiplication by $\epsilon_\varphi$. }
To show that $\theta_{\epsilon_\varphi}$ extends $\varphi$, it suffices to show that $\varphi(\mu^*) = \mu^* \cdot \epsilon_\varphi$ for each  $\mu \in F^+_E(T(v))$.   But again invoking Lemma~\ref{lemma:mu*lambda} we get 
\[\mu^* \cdot \epsilon_\varphi=\mu^*\sum_{\lambda \in F^+_E(T(v))} k_\lambda \lambda p^*_{v,r(\lambda)}=k_\mu p^*_{v,r(\mu)},\]
which  is precisely $\varphi(\mu^*)$.

For the second step,  let  $\psi: D\to X_v^E$ be a homomorphism of left $R$-modules, where $D$ is a left ideal of $R$ containing $I(v)$.   We  prove that the restriction map $\Hom_R(D, X_v^E)\to \Hom_R(I(v), X_v^E)$ is injective; the desired conclusion then follows  by invoking   Lemma~\ref{lemma:restriction}. Assume to the contrary that $\psi_1,\psi_2: D \to X_v^E$ are two different homomorphisms   with $\psi_1(x)=\psi_2(x)$ for all $x\in I(v)$. Since $\psi_1\not=\psi_2$, there exists $y\in D$ such that $\psi_1(y) \not= \psi_2(y)$ in $X_v^E$. Set
\[\psi_1(y)= \sum_{\lambda \in F^+_E(T(v))} k_\lambda \lambda p^*_{v,r(\lambda)}\quad\text{and}\quad \psi_2(y) =
 \sum_{\lambda \in F^+_E(T(v))}h_\lambda \lambda p^*_{v,r(\lambda)}.\]
 So, there exists $\gamma \in F^+_E(T(v))$ such that $k_{\gamma}\not= h_\gamma$, and thus 
\[k_\gamma  p_{v,r(\gamma)}^* \not= h_\gamma  p_{v,r(\gamma)}^*.\]
But by Lemma~\ref{lemma:mu*lambda} we have $\gamma^* \cdot\psi_1(y)  =   k_\gamma  p_{v,r(\gamma)}^*$, and similarly $\gamma^* \cdot \psi_2(y) =   h_\gamma  p_{v,r(\gamma)}^*$.  Since $\gamma^* \in I(v)$ we have  $\gamma^* y \in I(v)$. But then 
 $$k_\gamma p_{v,r(\gamma)}^* = \gamma^*  \cdot \psi_1(y) =  \psi_1(\gamma^* y)= \psi_2(\gamma^* y)=\gamma^* \cdot \psi_2(y) = h_\gamma p_{v,r(\gamma)}^* ,$$
 contrary to the displayed inequality.  
 
 Having established properties (1) and (2) we conclude by Proposition~\ref{prop:Baire} that       $X_v^E$ is injective.  So we have shown both that $Rv$ is an essential $R$-submodule of the left $R$-module $X_v^E$, and that $X_v^E$ is injective, thereby yielding that $X_v^E:= K[[F^+_E(T(v))]]^{sf}$ is the injective envelope of $Rv$. 
 \end{proof}

 %\begin{corollary}\label{cor:injhullofanysimpleleftideal}
%Let $E$ be any graph and $K$ any field.  Let $R$ denote $L_K(E)$. Let $N$ be any minimal left ideal of $R$.   Then the injective hull of $N$ 

 %\end{corollary}

As noted previously, any sink $w$ in a graph $E$ is necessarily a line point.  When $w$ is a sink, it is clear that $F^+_E(T(w)) = \Path(E)w$ and for any $\mu\in F^+_E(T(w))$ one has $p^*_{w,r(\mu)}=p^*_{w,w}=w$.  Theorem~\ref{injhulltheorem} then yields the following.

\begin{corollary}\label{sinkcorollary}
Let $E$ be an arbitrary graph, and $K$ any field.  Let  $w$ be a sink in $E$.  %Then $T(w)=w$, and 
Then the injective envelope of the simple left $L_K(E)$-ideal $L_K(E)w$ is the formal series left $L_K(E)$-module
$$X_w^E:= K[[F^+_E(w)]]^{sf}  \ = \  \{ \sum_{\lambda \in \Path(E)w} k_\lambda \lambda \ | \ k_\lambda\in K, \{s(\lambda): k_\lambda\not=0\} \ \mbox{is finite}\}.$$
%
%
%$K[[F^+_E(T(w))]]^{sf}$ of the formal series \[\sum_{\lambda \in \Path(E)w} k_\lambda \lambda\]
%such that $\{s(\lambda): k_\lambda\not=0\}$ is finite.
\end{corollary}
In particular, Corollary~\ref{sinkcorollary} generalizes, to all graphs and all sinks within these graphs, the ``formal series''  description of the injective envelope of the simple left $L_K(\mathcal{T})$-ideal $L_K(\mathcal{T})w$ given in \cite{AMT21}, where $\mathcal{T}$ is the Toeplitz graph and $w$ is the unique sink in $\mathcal{T}$.  Indeed, the corollary generalizes the description of  the injective envelope of the simple left ideal generated by a sink in any Leavitt path algebra associated to a finite graph with disjoint cycles given in \cite{AMT24}.

\medskip
We conclude the article by applying our main result to get a description of the injective envelope of simple left $L_K(E)$-modules associated with infinite emitters in $E$.

We recall that if $H$ is any hereditary subset  of $E^0$, a vertex $u\in E^0$ is called a \emph{breaking vertex of} $H$ \cite[Definition~2.4.4]{AAM} if $u$ is in the set
\[B_H:=\{u\in E^0\setminus H\mid u\in {\rm Inf}(E)\text{ and }0<|s^{-1}(u)\cap r^{-1}(E^0\setminus H)|<\infty\}. \]
For $u\in B_{H}$ we define the element $u^H$ of $L_K(E)$ by setting
\[
u^H:=u-\sum_{e\in s^{-1}(u)\cap r^{-1}(E^0\setminus H)}ee^*.
\]
For any subset $S\subseteq B_H$, we define $S^H\subseteq L_K(E)$ by setting $S^H=\{u^H\mid u\in S\}$.

In \cite{AR14},  
%Ara and Rangaswamy  produced a simple $L_K(E)$-module associated to each infinite emitter $v$ in $E$. The form of this simple module associated to $v$ depends on whether $v$ is or is not a breaking vertex for the hereditary and saturated set 
%\[H_v:=E^0\setminus \{u\in E^0\mid u\geq v\}.\]
%in $E$.
%This has to be changed: Ara and Ranga do not consider the case in wich there are infinitely many arrows starting in $v$ and ending outside $H_v$. We have to write:  
Ara and Rangaswamy  produced  simple $L_K(E)$-modules associated to two types of infinite emitters $v$ in $E$:  first, those for which $v$ is a breaking vertex for the hereditary and saturated set 
\[H_v:=E^0\setminus \{u\in E^0\mid u\geq v\}.\]
in $E$ (i.e., $v\in B_{H_v}$),  and second, those for which $r(s^{-1}(v))\subseteq H_v$.

\begin{corollary}\label{infemitterCor}
Let $v$ be an infinite emitter in $E$ of one of the two types described above.  Then we have an explicit description of the injective envelope of the simple module associated to $v$.
%presented in \cite{AR14}.
\end{corollary}

\begin{proof}

\underline{First type:}   Assume $v\in B_{H_v}$. Let $S:=B_{H_v}\setminus \{v\}$. Denote by $I(H_v, S^{H_v})$ the two-sided ideal of $L_K(E)$ generated by the subset $H_v\cup S^{H_v}\subseteq L_K(E)$.  By \cite[Theorem 2.4.15]{AAM} the quotient $L_K(E)/I(H_v, S^{H_v})$ is isomorphic to the Leavitt path algebra $L_K(E/(H_v,S))$ associated to the graph $E/(H_v,S)$, where $E/(H_v,S)$ is defined by setting
\begin{align*}
    \big(E/ (H_v,S)\big)^0&:=(E^0\setminus H_v)\sqcup \{v'\},\\
     \big(E/(H_v,S)\big)^1&:=\{e\in E^1\mid r(e)\notin H_v\}\sqcup \{e'\mid e\in E^1, r(e)=v\},
\end{align*}
extending the source and range maps in $E$ when appropriate, and in addition  setting $s(e')=s(e)$, $r(e')=v'$.
The vertex $v'$ is a sink in $E/(H_v,S)$. By Corollary~\ref{sinkcorollary}, the formal series left $L_K(E/(H_v,S))$-module 
$$X_{v'}^{E/(H_v,S)} := K[[F^+_{E/(H_v,S)}(v')]]^{sf}$$ is the injective envelope of the simple left $L_K(E/(H_v,S))$-module $\big(L_K(E/(H_v,S))\big)v'$. Through the projection
\[
L_K(E)\to L_K(E)/I(H_v, S^{H_v})\cong L_K(E/(H_v,S))
\]
both $X_{v'}^{E/(H_v,S)}$ and $\big(L_K(E/(H_v,S))\big)v'$ are also left $L_K(E)$-modules. By Lemma~\ref{lemma:injectivelocalunits}, $X_{v'}^{E/(H_v,S)}$ is the injective envelope of the simple $L_K(E)$-module $\big(L_K(E/(H_v,S))\big)v'$ in  $L_K(E)\Mod$.

\medskip

\underline{Second type:}  
%Assume $v\notin B_{H_v}$. This is not correct:  We have to write
Assume $r(s^{-1}(v))\subseteq H_v$. By \cite[Theorem 2.4.15]{AAM} the quotient $L_K(E)/I(H_v, B_{H_v}^{H_v})$ is isomorphic to the Leavitt path algebra $L_K(E/(H_v,B_{H_v}))=L_K(E/H_v)$ associated to the graph $E/(H_v,B_{H_v})=E/H_v$, where $E/H_v$ is defined by setting
\begin{align*}
    \big(E/H_v\big)^0&:=(E^0\setminus H_v),\\
     \big(E/H_v\big)^1&:=\{e\in E^1\mid r(e)\notin H_v\},
\end{align*}
and restricting the source and range maps in $E$ to $E/H_v$. 
The vertex $v$ is a sink in $E/H_v$. By Corollary~\ref{sinkcorollary}, the formal series left $L_K(E/H_v)$-module 
$$X_v^{E/H_v}:= K[[F^+_{E/H_v}(v)]]^{sf}$$ is the injective envelope of the simple left $L_K(E/H_v)$-module $\big(L_K(E/H_v)\big)v$. Through the projection
\[
L_K(E)\to L_K(E)/I(H_v, B_{H_v}^{H_v})\cong L_K(E/H_v)
\]
both $X_v^{E/H_v}$ and $\big(L_K(E/H_v)\big)v$ are also left $L_K(E)$-modules. By Lemma~\ref{lemma:injectivelocalunits}, $X_v^{E/H_v}$ is the injective envelope of the simple $L_K(E)$-module $\big(L_K(E/H_v)\big)v$ in  $L_K(E)\Mod$.
\end{proof}

%\R{AAA: Look at the following example}

\begin{example}
Consider the following graph $E$

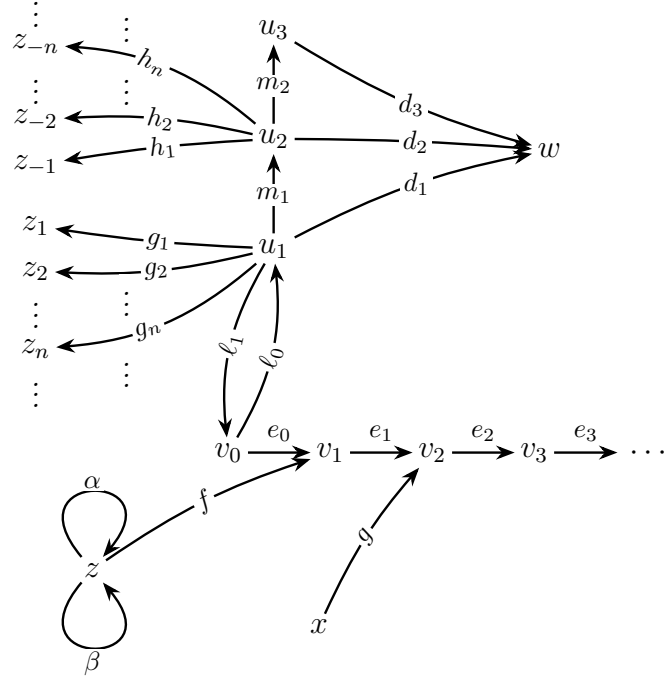
\begin{center}
\begin{tikzpicture}[
  >=Stealth,
  line width=0.9pt,
  vertex/.style={inner sep=1.2pt,outer sep=1pt},
  bigloopabove/.style={loop above,out=130,in=50,min distance=15mm,looseness=4},
  bigloopbelow/.style={loop below,out=230,in=310,min distance=15mm,looseness=4},
  edgelabel/.style={midway,fill=white,inner sep=1.2pt,font=\footnotesize}
]

% Vertici principali u_i e w
\node[vertex] (u3) at (3.60,6.15) {$u_3$};
\node[vertex] (u2) at (3.60,4.70) {$u_2$};
\node[vertex] (u1) at (3.60,3.25) {$u_1$};
\node[vertex] (w)  at (7.25,4.55) {$w$};

% Famiglia z_{-1}, z_{-2}, ..., z_{-n} associata a u_2
\node[vertex] (zm1) at (0.45,4.35) {$z_{-1}$};
\node[vertex] (zm2) at (0.45,4.95) {$z_{-2}$};
\node at (0.45,5.43) {$\vdots$};
\node[vertex] (zmn) at (0.45,5.95) {$z_{-n}$};
\node at (0.45,6.45) {$\vdots$};

% Famiglia z_1, z_2, ..., z_n associata a u_1
\node[vertex] (z1) at (0.45,3.55) {$z_1$};
\node[vertex] (z2) at (0.45,2.95) {$z_2$};
\node at (0.45,2.47) {$\vdots$};
\node[vertex] (zn) at (0.45,1.92) {$z_n$};
\node at (0.45,1.42) {$\vdots$};

% Frecce h_i da u_2 verso i vertici z_{-i}
\draw[->] (u2) to[bend right=3]
  node[edgelabel,pos=0.47,sloped] {$h_1$} (zm1);
\draw[->] (u2) to[bend right=8]
  node[edgelabel,pos=0.49,sloped] {$h_2$} (zm2);
\draw[->] (u2) to[bend right=18]
  node[edgelabel,pos=0.58,sloped] {$h_n$} (zmn);

% Puntini nella famiglia delle frecce h_i
\node[fill=white,inner sep=0.7pt] at (1.65,5.42) {$\vdots$};
\node[fill=white,inner sep=0.7pt] at (1.65,6.25) {$\vdots$};%era 6.05

% Frecce g_i da u_1 verso i vertici z_i
\draw[->] (u1) to[bend left=3]
  node[edgelabel,pos=0.47,sloped] {$g_1$} (z1);
\draw[->] (u1) to[bend left=8]
  node[edgelabel,pos=0.49,sloped] {$g_2$} (z2);
\draw[->] (u1) to[bend left=18]
  node[edgelabel,pos=0.58,sloped] {$g_n$} (zn);

% Puntini nella famiglia delle frecce g_i
\node[fill=white,inner sep=0.7pt] at (1.65,2.62) {$\vdots$};%era 2.42
\node[fill=white,inner sep=0.7pt] at (1.65,1.68) {$\vdots$};%era 1.78

% Frecce d_1,d_2,d_3 verso w
\draw[->] (u1) to[bend left=8]
  node[edgelabel,pos=0.53] {$d_1$} (w);
\draw[->] (u2) to[bend left=2]
  node[edgelabel,pos=0.51] {$d_2$} (w);
\draw[->] (u3) to[bend right=8]
  node[edgelabel,pos=0.51] {$d_3$} (w);
  
  %Frecce m_1, m_2, tra le u_i
   \draw[->] (u1) -- 
  node[edgelabel,pos=0.48,font=\footnotesize]
  {$m_1$} (u2);
  
   \draw[->] (u2) -- 
  node[edgelabel,pos=0.48,font=\footnotesize]
  {$m_2$} (u3);

% Catena v_0 -> v_1 -> v_2 -> v_3 -> ...
\node[vertex] (v0) at (3.00,0.55) {$v_0$};
\node[vertex] (v1) at (4.35,0.55) {$v_1$};
\node[vertex] (v2) at (5.70,0.55) {$v_2$};
\node[vertex] (v3) at (7.05,0.55) {$v_3$};
\coordinate (vr) at (8.15,0.55);

\draw[->] (v0) -- 
  node[pos=0.50,above,font=\footnotesize] {$e_0$} (v1);

\draw[->] (v1) -- 
  node[pos=0.50,above,font=\footnotesize] {$e_1$} (v2);

\draw[->] (v2) -- 
  node[pos=0.50,above,font=\footnotesize] {$e_2$} (v3);

\draw[->] (v3) -- 
  node[pos=0.50,above,font=\footnotesize] {$e_3$} (vr);
%\draw[->] (v0) -- node[edgelabel] {$e_0$} (v1);
%\draw[->] (v1) -- node[edgelabel] {$e_1$} (v2);
%\draw[->] (v2) -- node[edgelabel] {$e_2$} (v3);
%\draw[->] (v3) -- node[edgelabel] {$e_3$} (vr);
\node[right] at (vr) {$\cdots$};

% Due frecce opposte tra u_1 e v_0
\draw[->] (u1) to[bend right=18]
  node[edgelabel,pos=0.48,sloped] {$\ell_1$} (v0);
\draw[->] (v0) to[bend right=18]
  node[edgelabel,pos=0.48,sloped] {$\ell_0$} (u1);

% Vertice z con due loop alpha, beta e freccia f verso v_1
\node[vertex] (z) at (1.20,-1.00) {$z$};
\draw[->] (z) edge[bigloopabove]
  node[edgelabel] {$\alpha$} (z);
\draw[->] (z) edge[bigloopbelow]
  node[edgelabel] {$\beta$} (z);
\draw[->] (z) to[bend left=8]
  node[edgelabel,pos=0.50,sloped] {$f$} (v1);

% Vertice x e freccia g verso v_2
\node[vertex] (x) at (4.20,-1.75) {$x$};
\draw[->] (x) to[bend left=8]
  node[edgelabel,pos=0.50,sloped] {$g$} (v2);

\end{tikzpicture}
\end{center}
In the graph $E$ consider
\begin{itemize}
    \item the \emph{infinite} line point $v_1$;
    \item the \emph{finite} line point $u_3$;
    \item the sink $w$;
    \item the infinite emitter $u_1$ belonging to the set $B_{H_{u_1}}$ of breaking vertices of the hereditary saturated subset $H_{u_1}$ of $E^0$;
    \item the infinite emitter $u_2$ such that $r(s^{-1}(u_2))$ is contained in the hereditary and saturated subset $H_{u_2}$ of $E^0$.
\end{itemize}
Let $S_1:= B_{H_{u_1}}\setminus\{u_1\}$ and 
$S_2:=B_{H_{u_2}}$. 
Consider the graphs
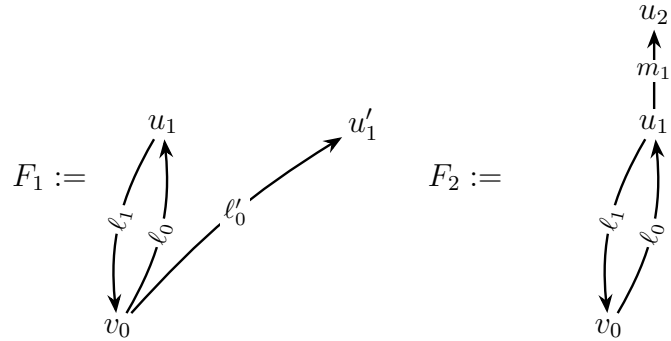
\begin{center}
\begin{tikzpicture}[
  >=Stealth,
  line width=0.9pt,
  vertex/.style={inner sep=1.2pt,outer sep=1pt},
  bigloopabove/.style={loop above,out=130,in=50,min distance=15mm,looseness=4},
  bigloopbelow/.style={loop below,out=230,in=310,min distance=15mm,looseness=4},
  edgelabel/.style={midway,fill=white,inner sep=1.2pt,font=\footnotesize}
]

% Vertici principali u_i e w
%\node[vertex] (u3) at (3.60,5.70) {$u_3$};
%\node[vertex] (u2) at (3.60,4.70) {$u_2$};
\node[vertex] (F1)  at (1.10,2.55) {$F_1:=$};
\node[vertex] (u1) at (2.60,3.25) {$u_1$};
\node[vertex] (w)  at (5.25,3.25) {$u_1'$};
\node[vertex] (v0) at (2.00,0.55) {$v_0$};

% Due frecce opposte tra u_1 e v_0
\draw[->] (u1) to[bend right=18]
  node[edgelabel,pos=0.48,sloped] {$\ell_1$} (v0);
\draw[->] (v0) to[bend right=18]
  node[edgelabel,pos=0.48,sloped] {$\ell_0$} (u1);
  
  \draw[->] (v0) to[bend left=8]
  node[edgelabel,pos=0.53] {$\ell_0'$} (w);

\end{tikzpicture}
\quad
\begin{tikzpicture}[
  >=Stealth,
  line width=0.9pt,
  vertex/.style={inner sep=1.2pt,outer sep=1pt},
  bigloopabove/.style={loop above,out=130,in=50,min distance=15mm,looseness=4},
  bigloopbelow/.style={loop below,out=230,in=310,min distance=15mm,looseness=4},
  edgelabel/.style={midway,fill=white,inner sep=1.2pt,font=\footnotesize}
]

% Vertici principali u_i e w
%\node[vertex] (u3) at (3.60,5.70) {$u_3$};
%\node[vertex] (u2) at (3.60,4.70) {$u_2$};
\node[vertex] (F2)  at (1.10,2.55) {$F_2:=$};
\node[vertex] (u1) at (3.60,3.25) {$u_1$};
\node[vertex] (v0) at (3.00,0.55) {$v_0$};
\node[vertex] (u2) at (3.60,4.70) {$u_2$};

% Due frecce opposte tra u_1 e v_0
\draw[->] (u1) to[bend right=18]
  node[edgelabel,pos=0.48,sloped] {$\ell_1$} (v0);
\draw[->] (v0) to[bend right=18]
  node[edgelabel,pos=0.48,sloped] {$\ell_0$} (u1);
  
%  \draw[->] (u1) -- 
%  node[pos=0.50,above,font=\footnotesize] {$m_1$} (u2);
  
%  \draw[->] (u1) to[bend right=18]
%  node[edgelabel,pos=0.48,sloped] {$\ell_1$} (u2);
% 
 \draw[->] (u1) -- 
  node[edgelabel,pos=0.48,font=\footnotesize]
  {$m_1$} (u2);

\end{tikzpicture}
\end{center}
Then 
$L_K(F_1)\cong L_K(E)/I(H_{u_1}, S_1^{H_{u_1}})$ and $L_K(F_2)\cong L_K(E)/I(H_{u_2}, S_2^{H_{u_2}})$. Each left $L_K(F_1)$-module and each left $L_K(F_2)$-module are also left $L_K(E)$-modules through the $K$-algebra projections $L_K(E)\to L_K(F_1)$, and $L_K(E)\to L_K(F_2)$.

Denote by $\{x,y\}^*$ the set of all words (the empty word included) in the alphabet $\{x,y\}$. 
To each of $v_1$, $u_3$, $w$, $u_1$, and $u_2$ is associated a simple left $L_K(E)$-module: $L_K(E)v_1$, $L_K(E)u_3$, $L_K(E)w$, $L_K(F_1)u'_1$, and $L_K(F_2)u_2$. 
Their injective envelopes are, respectively,  the formal series left $L_K(E)$-modules
\begin{align*}
    X_{v_1}^E&=\Bigg\{\sum_{\mu\in F^+_E(T(v_1))}k_\mu \mu p^*_{v_1,r(\mu)}\mid k_\mu\in K\Bigg\},
    % \\
    % &=\Big\{\sum_{\pi(\alpha,\beta)\in \{\alpha,\beta\}^*} k_{\pi(\alpha,\beta)}\pi(\alpha,\beta)f+
    % \sum_{\pi(\ell_0,\ell_1)\in \{\ell_),\ell_1\}^*}k_{\pi(\ell_0,\ell_1)}\pi(\ell_0,\ell_1)e_0+\\
    % &+
    % k_g gp^*_{v_2,v_1}+\sum_{i\geq 1}k_ip^*_{v_1,v_i}\mid k\in K\Big\}
\end{align*}
\begin{align*}
    X_{u_3}^E&=\Bigg\{\sum_{\mu\in F^+_E(T(u_3))}k_\mu \mu p^*_{u_3,r(\mu)}\mid k_\mu\in K\Bigg\},
    % \\
    % &=\Big\{k_{u_3}u_3+k_{m_2}m_2+\sum_{\pi(\ell_0,\ell_1)\in \{\ell_),\ell_1\}^*}k_{\pi(\ell_0,\ell_1)m_1m_2}\pi(\ell_0,\ell_1)m_1m_2+k_wd_3^*+\\
    % &+k_{d_2}d_2d_3^*+
    % +\sum_{\pi(\ell_0,\ell_1)\in \{\ell_),\ell_1\}^*}k_{\pi(\ell_0,\ell_1)m_1d_2}\pi(\ell_0,\ell_1)m_1d_2d^*_3+\\
    % &+\sum_{\pi(\ell_0,\ell_1)\in \{\ell_),\ell_1\}^*}k_{\pi(\ell_0,\ell_1)d_1}\pi(\ell_0,\ell_1)d_1d_3^*\mid
    %  k\in K\Big\}
\end{align*}
\begin{align*}
    X_w^E
    % &=\{ \sum_{\mu \in \Path(E)w} k_\mu \mu \ | \ k_\mu\in K, \{s(\mu): k_\mu\not=0\} \ \mbox{ is finite}\}\\
    &=\Bigg\{\sum_{\mu\in \Path(E)w}k_\mu \mu\mid k_\mu\in K\Bigg\},
\end{align*}
\begin{align*}
    X_{u_1'}^E&=\Bigg\{\sum_{\mu\in\Path(F_1)u'_1}k_\mu \mu\mid k_\mu\in K\Bigg\}, \  \mbox{and}
\end{align*}
\begin{align*}
    X_{u_2}^E&=\Bigg\{\sum_{\mu\in\Path(F_2)u_2}k_\mu \mu\mid k_\mu\in K\Bigg\}.
\end{align*}
\end{example}

\smallskip

\section{Appendix}\label{sec:appendix}

In this Appendix we present the description of the aforementioned Cuntz-Krieger $E$-family inside ${\rm End}_K(K[[F^+_E(T(v))]])$ when $v$ is a line point in $E$.   
%that The left $R$-premodule structure of $K[[F^+_E(T(v))]]$ when $v$ is a line point 

Let $u\in E^0$ and $e\in E^1$. We define three $K$-linear endomorphisms $P_u$, $S_e$,
$S_{e^*}$ of the $K$-vector space $K[[F^+_E(T(v))]]$ of all functions from   $F^+_E(T(v))$ to $K$.
Let $\mathfrak p\in K[[F^+_E(T(v))]]$. Then for each $\mu\in F^+_E(T(v))$ we define
\[
P_u(\mathfrak p)(\mu):=\begin{cases}
    \mathfrak p(\mu)  & \text{if }s(\mu)=u, \\
     0 & \text{otherwise};
\end{cases}
\]

\[
S_e(\mathfrak p)(\mu):=\begin{cases}\begin{cases}
    \mathfrak p(\mu')  & \text{if }\mu=e\mu', \\
     0 & \text{otherwise};
\end{cases} & \text{if }s(e)\not\in T(v)\\
\begin{cases}
    \mathfrak p(r(e))  & \text{if } \mu=s(e), \\
     0 & \text{otherwise};
\end{cases}
 & \text{if }s(e)\in T(v)
\end{cases}
\]

\[
S_{e^*}(\mathfrak p)(\mu):=\begin{cases}\begin{cases}
    \mathfrak p(e\mu)  & \text{if }s(\mu)=r(e), \\
     0 & \text{otherwise};
\end{cases} & \text{if }s(e)\not\in T(v)\\
\begin{cases}
    \mathfrak p(s(e))  & \text{if } \mu=r(e), \\
     0 & \text{otherwise};
\end{cases}
 & \text{if }s(e)\in T(v).
\end{cases}
\]
Clearly $P_u$, $S_e$, and $S_{e^*}$ are $K$-linear. If
\[\mathfrak p=\sum_{\mu\in F^+_E(T(v))}k_\mu \mu p^*_{v,r(\mu)}\]
then
\begin{align*}
P_u(\mathfrak p)&=\sum_{
\substack{\mu\in F^+_E(T(v))\\ s(\mu)=u}
}k_\mu \mu p^*_{v,r(\mu)}=\sum_{\mu\in F^+_E(T(v))}k_\mu u\mu p^*_{v,r(\mu)}\\
S_e(\mathfrak p)&=\begin{cases}
\displaystyle\sum_{\substack{
\mu\in F^+_E(T(v))\\
\mu=e\mu'}
}k_{\mu'} \mu p^*_{v,r(\mu)}  =
   \sum_{
\mu\in F^+_E(T(v))}
k_{\mu} e\mu p^*_{v,r(\mu)}   & \text{if }s(e)\notin T(v), \\
k_{r(e)}p^*_{v,s(e)}= k_{r(e)}ep^*_{v,r(e)}    & \text{if }s(e)\in T(v).
\end{cases}
\\
S_{e^*}(\mathfrak p)&=\begin{cases}
   \sum_{\substack{
\mu\in F^+_E(T(v))\\
s(\mu)=r(e)
}}k_{e\mu} \mu p^*_{v,r(\mu)}   & \text{if }s(e)\notin T(v), \\
 k_{s(e)}p^*_{v,r(e)}= e^*p^*_{v,s(e)}    & \text{if }s(e)\in T(v)
\end{cases}
\end{align*}
Less formally,  the effect of each of the $K$-endomorphisms $P_u, S_e,$ and $S_{e^*}$ corresponds, respectively,  to multiplying by $u$, $e$ and $e^*$ each of the summands of the formal series  $\mathfrak p=\sum_{\mu\in F^+_E(T(v))}k_\mu \mu p^*_{v,r(\mu)}$.

\textbf{Claim}: The subset $\{P_u, S_e, S_{e^*}\mid u\in E^0, e\in E^1\}$ forms a Cuntz-Krieger $E$-family in $\End(K[[F^+_E(T(v))]])$.

\textbf{Proof of the claim}: We suppress the composition operator $\circ$ in the following equations (function composition is written right to left, so that $fg$ means ``first $g$, then $f$'');   $\delta$ denotes the Kronecker delta. By definition it must be checked that:
\begin{enumerate}
\item $P_uP_w=\delta_{u,w}P_u$ for all $u,w\in E^0$;
\item $P_{s(e)}S_e=S_e=S_eP_{r(e)}$ for all $e\in E^1$;
\item $P_{r(e)}S_{e^*}=S_{e^*}=S_{e^*}P_{s(e)}$ for all $e\in E^1$;
\item $S_{f^*}S_e=\delta_{e,f}P_{r(e)}$ for all $e,f\in E^1$;
\item $\sum_{e\in s^{-1}(u)} S_eS_{e^*}=P_u$ for any regular vertex $u\in E^0$.
\end{enumerate}

\begin{align*}(1)\quad P_u(P_w(\mathfrak p))(\mu)&=\begin{cases}
 P_w(\mathfrak p)(\mu)   & \text{if }s(\mu)=u, \\
    0  & \text{otherwise}
\end{cases}\\
&=\begin{cases}
\mathfrak p(\mu)   & \text{if }s(\mu)=w=u, \\
    0  & \text{otherwise}
\end{cases}\\
&=\delta_{u,w}P_u(\mathfrak p)(\mu)\\
\end{align*}
\text{ for each $\mu\in F^+_E(T(v))$.}

\begin{align*}(2)\quad P_{s(e)}(S_e(\mathfrak p))(\mu)&=\begin{cases}
 S_e(\mathfrak p)(\mu)   & \text{if }s(\mu)=s(e), \\
    0  & \text{otherwise}
\end{cases}\\
&=S_e(\mathfrak p)(\mu)\quad\text{and}
\end{align*}
\begin{align*}\quad S_e( P_{r(e)}(\mathfrak p))(\mu)&=
\begin{cases}\begin{cases}
    (P_{r(e)}(\mathfrak p))(\mu')  & \text{if }\mu=e\mu', \\
     0 & \text{otherwise};
\end{cases} & \text{if }s(e)\not\in T(v)\\
\begin{cases}
     (P_{r(e)}(\mathfrak p))(r(e))  & \text{if } \mu=s(e), \\
     0 & \text{otherwise};
\end{cases}
 & \text{if }s(e)\in T(v)
\end{cases}\\
&=
\begin{cases}\begin{cases}
    \mathfrak p(\mu')  & \text{if }\mu=e\mu', \\
     0 & \text{otherwise};
\end{cases} & \text{if }s(e)\not\in T(v)\\
\begin{cases}
     \mathfrak p(r(e))  & \text{if } \mu=s(e), \\
     0 & \text{otherwise};
\end{cases}
 & \text{if }s(e)\in T(v)
\end{cases}\\
&=S_e(\mathfrak p)(\mu)
\end{align*}

\begin{align*}(3)\quad P_{r(e)}(S_{e^*}(\mathfrak p))(\mu)&=\begin{cases}
 S_{e^*}(\mathfrak p)(\mu)   & \text{if }s(\mu)=r(e), \\
    0  & \text{otherwise}
\end{cases}\\
&=S_{e^*}(\mathfrak p)(\mu)\quad\text{and}
\end{align*}

\begin{align*}\quad S_{e^*}( P_{s(e)}(\mathfrak p))(\mu)&=
\begin{cases}\begin{cases}
    P_{s(e)}(\mathfrak p)(e\mu)  & \text{if }r(e)=s(\mu), \\
     0 & \text{otherwise};
\end{cases} & \text{if }s(e)\not\in T(v)\\
\begin{cases}
     P_{s(e)}(\mathfrak p)(s(e))  & \text{if } \mu=r(e), \\
     0 & \text{otherwise};
\end{cases}
 & \text{if }s(e)\in T(v)
\end{cases}\\
&=\begin{cases}\begin{cases}
    \mathfrak p(e\mu)  & \text{if }r(e)=s(\mu), \\
     0 & \text{otherwise};
\end{cases} & \text{if }s(e)\not\in T(v)\\
\begin{cases}
     \mathfrak p(s(e))  & \text{if } \mu=r(e), \\
     0 & \text{otherwise};
\end{cases}
 & \text{if }s(e)\in T(v)
\end{cases}\\
&=S_{e^*}(\mathfrak p)(\mu)
\end{align*}
\begin{align*}(4)\quad S_{f^*}(S_{e}(\mathfrak p))(\mu)&=
\begin{cases}\begin{cases}
    S_{e}(\mathfrak p)(f\mu)  & \text{if }r(f)=s(\mu), \\
     0 & \text{otherwise};
\end{cases} & \text{if }s(f)\not\in T(v)\\
\begin{cases}
     S_{e}(\mathfrak p)(s(f))  & \text{if } \mu=r(f), \\
     0 & \text{otherwise};
\end{cases}
 & \text{if }s(f)\in T(v)
\end{cases}\\
\text{If $s(e)\notin T(v)$ we get}\\
&=\begin{cases}\begin{cases}
\mathfrak p(\mu)& \text{if } e=f, r(f)=s(\mu), \\
0 & \text{otherwise};
\end{cases} & \text{if }s(f)\not\in T(v)\\
0& \text{if }s(f)\in T(v)
\end{cases}\\
\text{If $s(e)\in T(v)$ we get}\\
&=\begin{cases}
0& \text{if }s(f)\notin T(v)\\
\begin{cases}
\mathfrak p(r(e))& \text{if } s(e)=s(f), r(f)=\mu, \\
0 & \text{otherwise};
\end{cases} & \text{if }s(f)\in T(v)
\end{cases}
\end{align*}

Observe that if $s(e)=s(f)\in T(v)$, then $e=f=p_{s(e),r(e)}$. Therefore, resuming, we have
\begin{align*}\quad S_{f^*}(S_{e}(\mathfrak p))(\mu)&=
\begin{cases}
\mathfrak p(\mu)   & \text{if }e=f,\ s(\mu)=r(f)\\
0  & \text{otherwise}
\end{cases}\\
&=\delta_{e,f}P_{r(f)}(\mathfrak p)(\mu).
\end{align*}
Finally, assume that $u\in E^0$ is a regular vertex. If $u\in T(v)$, necessarily $s^{-1}(u)$ contains only one edge of the form $p_{u,u'}$.  Then
\begin{align*}(5)\quad \sum_{e\in s^{-1}(u)}S_{e}(S_{e^*}(\mathfrak p))(\mu)&=
\begin{cases}\begin{cases}
    S_{\varepsilon^*}(\mathfrak p)(\mu')  & \text{if }\substack{\mu=\varepsilon \mu'\\ \varepsilon\in s^{-1}(u)}, \\
     0 & \text{otherwise};
\end{cases} & \text{if }u\not\in T(v)\\
\begin{cases}
     S_{p_{u,u'}^*}(\mathfrak p)(u')  & \text{if } \mu=u, \\
     0 & \text{otherwise};
\end{cases}
 & \text{if }u\in T(v)
\end{cases}\\
&=
\begin{cases}\begin{cases}
    \mathfrak p(\varepsilon\mu')=\mathfrak p(\mu)  & \text{if }\substack{\mu=\varepsilon \mu'\\ \varepsilon\in s^{-1}(u)}, \\
     0 & \text{otherwise};
\end{cases} & \text{if }u\not\in T(v)\\
\begin{cases}
     \mathfrak p(u)  & \text{if } \mu=u, \\
     0 & \text{otherwise};
\end{cases}
 & \text{if }u\in T(v)
\end{cases}
\end{align*}
Therefore, resuming, we have
\begin{align*}\quad \sum_{e\in s^{-1}(u)}S_{e}(S_{e^*}(\mathfrak p))(\mu)&=
\begin{cases}
\mathfrak p(\mu)   & \text{if }s(\mu)=u\\
0  & \text{otherwise}
\end{cases}\\
&=P_u(\mathfrak p)(\mu).
\end{align*}
 Steps (1) through (5) have established the Claim. By the Universal Property of $R=L_K(E)$ (see e.g., \cite[Remark 1.2.5]{AAM}), there exists a $K$-algebra homomorphism $R\to \End(K[[F^+_E(T(v))]])$, which thereby endows $K[[F^+_E(T(v))]]$ with the structure of left $R$-premodule.

\end{document}